\documentclass[12pt]{article}
\usepackage{amssymb,amstext}
\usepackage{amsmath,amsthm}
\usepackage[latin1]{inputenc} 
\usepackage[T1]{fontenc}      
\usepackage{graphicx}
\usepackage{eurosym}
\usepackage{color}

\usepackage[pdftex]{hyperref}

\usepackage{textcomp}         
\date{July 19, 2016}
\newtheorem{theorem}{Theorem}
\newtheorem{corollary}[theorem]{Corollary}

\newtheorem{lemma}[theorem]{Lemma}
\newtheorem{proposition}[theorem]{Proposition}
\newtheorem{definition}[theorem]{Definition}
\newtheorem{remark}[theorem]{Remark}
\newtheorem{assumption}[theorem]{Assumption}

\newtheorem{example}[theorem]{Example}
\newtheorem{specialcase}[theorem]{Special case}

\newcommand{\R}{\mathbb{R}}

\newcommand{\xad}{x_\alpha^\delta}
\newcommand{\xa}{x_\alpha}

\newcommand{\xdag}{x^\dagger}
\newcommand{\yd}{y^\delta}

\title{\bf Conditional stability versus ill-posedness for operator equations with monotone operators in Hilbert space}

\author{{\sc Radu Ioan Bo\c{t}}\thanks{Faculty of Mathematics, University of Vienna, Oskar-Morgenstern-Platz 1, A-1090 Vienna, {\sc Austria},  Email:$\;$\texttt{radu.bot\,@\,univie.ac.at}\,.} \,\,and\, {\sc Bernd Hofmann}\thanks{Faculty of Mathematics, Technische Universit\"at Chemnitz,  D-09107 Chemnitz,
{\sc Germany},  Email:$\;$\texttt{bernd.hofmann\,@\,mathematik.tu-chemnitz.de}\,.}}

\begin{document}

\maketitle

\vspace{-0.65cm}

\begin{abstract}
In the literature on singular perturbation (Lavrentiev regularization) for the stable approximate solution of operator equations with monotone operators in the Hilbert space the phenomena of conditional stability
and local well-posedness and ill-posedness are rarely investigated. Our goal is to present some studies which try to bridge this gap.
So we discuss the impact of conditional stability on error estimates and convergence rates for the Lavrentiev regularization and distinguish for linear problems well-posedness and ill-posedness in a specific manner motivated by a saturation result. The role of the regularization error in the noise-free case, called bias, is a crucial point in the paper for nonlinear and linear problems. In particular, for linear operator equations general convergence rates, including logarithmic rates, are derived by means of the method of approximate source conditions. This allows us to extend well-known convergence rates results for the Lavrentiev regularization that were based on general source conditions to the case of non-selfadjoint linear monotone forward operators for which general source conditions fail.
Examples presenting the self-adjoint multiplication operator as well as the non-selfadjoint fractional integral operator and Ces\`aro operator illustrate the theoretical results. Extensions to the nonlinear case under specific conditions on the nonlinearity structure complete the paper.
\end{abstract}

\vspace{0.3cm}

{\parindent0em {\bf MSC2010 subject classification:}
47A52}, 65F22, 47H05, 65J22, 65J15

\vspace{0.3cm}

{\parindent0em {\bf Keywords:}
Conditional stability, local well-posedness, local ill-posedness, Hilbert space, equations with monotone operators, linear and nonlinear operator equations, Lavrentiev regularization, Tikhonov regularization, convergence rates, H\"older and logarithmic rates, approximate source conditions, solution smoothness, multiplication operator, fractional integration operator,
Ces\`aro operator.

\section{Introduction}\label{s1}
\setcounter{equation}{0}
\setcounter{theorem}{0}

If $F: \mathcal{D}(F) \subseteq X \to Y$ denotes a sufficiently smooth and possibly nonlinear operator mapping between Hilbert spaces $X$ and $Y$ with norms $\|\cdot\|$, then it is not always trivial to find in a stable manner the solution $\xdag \in \mathcal{D}(F)$ to the operator equation
\begin{equation} \label{eq:nonlinopeq}
F(x)=y
\end{equation}
with the exact right-hand side $y=F(\xdag)$ when only noisy data $\yd$ obeying the deterministic noise model
\begin{equation} \label{eq:noise}
\|y-\yd\| \leq\delta
\end{equation}
with noise level $\delta>0$ are available. Even if (\ref{eq:nonlinopeq}) has $\xdag$ as the unique solution, a least squares approach $$\|F(x)-\yd\|^2 \to \min, \qquad \mbox{subject to}\quad x \in \mathcal{D}(F),$$ is not always successful if the Hilbert space is infinite dimensional.
Then the least squares minimizers need not exist and if they exist their convergence to $\xdag$ in the norm of $X$ as $\delta \to 0$ can only be expected if the operator equation is locally well-posed at $\xdag$.
In this context, we recall the following definition introduced in \cite[Definition~2]{HofSch94}.

\begin{definition}  \label{def:posed}
The equation (\ref{eq:nonlinopeq}) is called {\sl locally well-posed} at the solution point $\xdag \in \mathcal{D}(F)$ if there is a ball $\mathcal{B}_r(\xdag)$ with radius $r>0$ and center $\xdag$ such that for every sequence $\{x_k\}_{k=1}^\infty \subset \mathcal{B}_r(\xdag)\cap \mathcal{D}(F)$ the convergence of images $\lim \limits_{k \to \infty} \|F(x_k)-F(\xdag)\|=0$ implies the convergence of the preimages  $\lim \limits_{k \to \infty} \|x_k-\xdag\|=0$. Otherwise it is called {\sl locally ill-posed}.
\end{definition}

In particular if the equation (\ref{eq:nonlinopeq}) is a model of an inverse and therefore mostly ill-posed problem, it makes sense to exploit a singularly perturbed auxiliary problem to equation (\ref{eq:nonlinopeq}), which is automatically locally well-posed. The most prominent such approach is the Tikhonov regularization, where in the simplest case (cf.~\cite[Chapt.~10]{EHN96}) stable approximate solutions $\xad \in \mathcal{D}(F)$ solve the extremal problem
\begin{equation} \label{eq:Tik}
\|F(x)-\yd\|^2+\|x-\bar x\|^2 \to \min, \qquad \mbox{subject to} \quad x \in \mathcal{D}(F),
\end{equation}
with regularization parameter $\alpha>0$ and reference element (initial guess) $\bar x \in X$.
Variants of Tikhonov regularization, however, are also helpful and advantageous (cf., e.g., \cite{ChengHofLu14,ChengYam00} and \cite[\S 6.2]{HofYam10}) if (\ref{eq:nonlinopeq}) is locally well-posed in the sense that a {\sl conditional stability estimate} of the form
\begin{equation} \label{eq:condstab}
\|x-\xdag\| \le \varphi(\|F(x)-F(\xdag)\|) \qquad \mbox{for all} \quad x \in \mathcal{D}(F) \cap Q
\end{equation}
applies, with some set $Q \subset X$ containing $\xdag$ and some {\sl concave} index function $\varphi$, where we call
$\varphi:[0,\infty) \to [0,\infty)$ {\sl index function} if it is continuous, strictly increasing and satisfies the condition $\varphi(0)=0$. Then the method ensures convergence and rates of the approximate solutions
when the regularized solutions are embedded in the stability region $\mathcal{D}(F) \cap Q$.

The focus of this paper is on the specific situation of an operator equation (\ref{eq:nonlinopeq}) with \linebreak $Y=X,$ $\mathcal{D}(F)=X$, and {\sl monotone} operators $F$  as characterized by the following assumption.
Mostly, {\sl infinite dimensional} Hilbert spaces $X$ will be under consideration, but for examples also finite dimensional cases shall be exploited.

 \begin{assumption} \label{ass:ass0}
Consider the operator equation (\ref{eq:nonlinopeq}) with solution $\xdag \in X$  under the auspices that
\begin{enumerate}
\item[(a)] $X$ is a real separable Hilbert space with norm $\|\cdot\|$ and inner product $\langle \cdot,\cdot \rangle$ and
\item[(b)] $F: X \to X$ is a monotone operator, i.e.
\begin{equation}\label{eq:Fmon}
 \langle F(x)-F(\tilde{x}),x-\tilde{x}\rangle \geq 0 \qquad \mbox{for all} \quad x,\tilde{x}\in X,
 \end{equation}
which is moreover hemicontinuous and hence maximally monotone.
\end{enumerate}
\end{assumption}

Under Assumption~\ref{ass:ass0} there occur well-posed and ill-posed situations. The best situation of global well-posedness is characterized by {\sl strong monotonicity}
\begin{equation} \label{eq:Fmonstrongglob}
 \langle F(x)-F(\tilde x),x-\tilde x\rangle \geq C\,\|x-\tilde x\|^2 \qquad \mbox{for all} \quad x,\tilde x \in X,
\end{equation}
with some constant $C>0$, which implies the {\sl coercivity} condition
\begin{equation} \label{eq:coercive}
\lim \limits _{\|x\| \to \infty} \frac{\langle F(x),x \rangle}{\|x\|}=\infty.
\end{equation}

\begin{proposition} \label{pro:BrowderMinty}
Under the requirements of Assumption~\ref{ass:ass0} strengthened by the condition (\ref{eq:Fmonstrongglob}) the equation (\ref{eq:nonlinopeq}) is uniquely solvable  in $X$ for all $y \in X$, and
the solutions are Lipschitz continuous with respect to the data, i.e. the inverse operator $F^{-1}: X \to X$ is well-defined with
\begin{equation} \label{eq:ideal}
\|F^{-1}(y)-F^{-1}(\tilde y)\| \le \frac{1}{C}\,\|y-\tilde y\| \qquad \mbox{for all} \quad y,\tilde y \in X.
\end{equation}
\end{proposition}

\begin{proof}
The Browder-Minty theorem ensures under the supposed conditions that $F$ is surjective and due to (\ref{eq:Fmonstrongglob}) even bijective.
Also from (\ref{eq:Fmonstrongglob}) we have for all $x, \tilde x \in X$
$$C\,\|x-\tilde x\|^2 \le \langle F(x)-F(\tilde x),x-\tilde x\rangle \le \|x-\tilde x\|\,\| F(x)-F(\tilde x)\|,$$
which yields (\ref{eq:ideal}) and completes the proof.
\end{proof}

However, there are a lot of examples for inverse problems occurring in natural sciences, engineering, and finance, where estimates of the form (\ref{eq:ideal}) fail and operator equations (\ref{eq:nonlinopeq}) with monotone forward operators $F$ have to be solved in a stable approximate manner. Due to the smoothing character of $F$ in these cases, local ill-posedness  must be conjectured, and we have so-called operator equations of the first kind. Compact monotone operators $F$ are typical for that situation. Then $F$ obeys (\ref{eq:Fmon}), but fails to satisfy inequalities of the form (\ref{eq:Fmonstrongglob}). Examples of ill-posed problems in integral and
differential equations under monotone forward operators are, for example, presented in \cite[Section~1.3]{AlbRya06} and \cite[Section~5]{HKR16}.
 However, for all $\alpha>0$ the associated equations of the second kind $G(x)=y$ with $G(x):=F(x)+\alpha I$ are strongly monotone and hence locally well-posed everywhere, because we have
$$\langle G(x)-G(\tilde x),x-\tilde x\rangle =\langle F(x)-F(\tilde x),x-\tilde x\rangle + \alpha\,\|x-\tilde x\|^2 \geq \alpha\,\|x-\tilde x\|^2 \;\;\mbox{for all} \;\; x,\tilde{x}\in X.$$
This gives a substantial motivation for using singular perturbations
for the stable approximate solution of equation (\ref{eq:nonlinopeq}) also here. Due to the maximal monotonicity of $F$ the simpler
Lavrentiev regularization (cf.~the seminal monograph \cite{Lavrentiev67} as well as the more recent works \cite{AlbRya06,LiuNash96}) is applicable, where stable approximate solutions $\xad \in X$ solve the operator equation
\begin{equation} \label{eq:lav}
F(\xad)+\alpha(\xad-\bar{x})=\yd,
\end{equation}
with regularization parameter $\alpha>0$ and reference element $\bar x \in X$. Such approach is also helpful if Proposition~\ref{pro:BrowderMinty} is not applicable, because coercivity (\ref{eq:coercive}) fails or  well-posedness at $\xdag$ takes place only in a local sense.
The latter is the case if $F$ is {\sl strongly monotone} in a neighbourhood of $\xdag$, i.e.~the locally relaxed version of (\ref{eq:Fmonstrongglob}),
\begin{equation} \label{eq:Fmonstrongloc}
 \langle F(x)-F(\xdag),x-\xdag\rangle \geq C\,\|x-\xdag\|^2 \qquad \mbox{for all} \quad x \in \mathcal{B}_r(\xdag),
\end{equation}
with some radius $r>0$ and some constant $C>0$ is valid,
 or if $F$ is {\sl uniformly monotone} in a neighbourhood of $\xdag$, i.e.
\begin{equation} \label{eq:Fmongenloc}
 \langle F(x)-F(\xdag),x-\xdag\rangle \geq \zeta(\|x-\xdag\|) \qquad \mbox{for all} \quad x \in \mathcal{B}_r(\xdag)
\end{equation}
holds with some radius $r>0$ and some index function $\zeta$. In both situations we have local well-posedness at $\xdag$, and we refer for examples to the monograph \cite{Zeidler90} and also to papers like \cite{Bes10}.

\begin{proposition} \label{pro:condstab}
Let the inequality (\ref{eq:Fmongenloc}) hold with an index function $\zeta$ of the form $\zeta(t)=\theta(t)\,t,\;t>0,$ such that $\theta$ is a convex index function. Then the condition (\ref{eq:Fmongenloc}) of local uniform monotonicity is a conditional stability estimate of the form (\ref{eq:condstab}) with $\mathcal{D}(F)=X,\;Q=\mathcal{B}_r(\xdag),$ and the concave index function $\varphi(t)=\theta^{-1}(t),\;t>0$.
This implies that the operator equation (\ref{eq:nonlinopeq}) is locally well-posed at the solution point $\xdag$.
Evidently, these assertions apply for the local strong monotonicity (\ref{eq:Fmonstrongloc}) yielding (\ref{eq:condstab}) with $\varphi(t)=\frac{1}{C}\,t$.
\end{proposition}
\begin{proof}
For $x \in \mathcal{B}_r(\xdag)$ we can estimate from (\ref{eq:Fmongenloc}) as
$$\zeta(\|x-\xdag\|)=\theta(\|x-\xdag\|)\,\|x-\xdag\|\le  \langle F(x)-F(\xdag),x-\xdag\rangle \le  \|F(x)-F(\xdag)\|\,\|x-\xdag\|$$
and hence $\|x-\xdag\| \le \theta^{-1}(\|F(x)-F(\xdag)\|)$, where $\theta^{-1}$ is a concave index function which plays the role of $\varphi$ in (\ref{eq:condstab}).
The special case of local strong monotonicity (\ref{eq:Fmonstrongloc}) applies here with $\theta(t)=C\,t$.
\end{proof}

\begin{remark} \label{rem:convex}
{\rm In very specific cases, see Example~\ref{ex:power} below, the function $\theta$ in Proposition~\ref{pro:condstab} can also be concave such that $\theta^{-1}$ is a convex index function. Then, surprisingly,
with $0<\kappa<1$ the convergence rate (\ref{eq:Hrate}) in Corollary~\ref{cor:condstab} below can be overlinear as $\mathcal{O}(\delta^{1/\kappa})$.
}\end{remark}

A special case of (\ref{eq:nonlinopeq}) taking into account Assumption~\ref{ass:ass0} is characterized by forward operators $A \in \mathcal{L}(X)$ instead of $F$, where $\mathcal{L}(X)$ denotes the Banach space
of bounded linear operators $A:X \to X$ and $\|A\|$ indicates the corresponding operator norm.
So we consider in this case linear operator equations
\begin{equation} \label{eq:linopeq}
Ax=y
\end{equation}
under the noise model (\ref{eq:noise}), where $A$ is monotone (accretive), i.e.
\begin{equation}\label{eq:Amonotone}
 \langle Ax,x\rangle \geq 0 \qquad \mbox{for all} \quad x \in X.
\end{equation}
Note that for all such operators $A$ and all $\alpha>0$ the properties
\begin{equation} \label{eq:monprop}
(A+\alpha I)^{-1} \in \mathcal{L}(X),\qquad \|(A+\alpha I)^{-1}A\| \le 1,
\end{equation}
and
\begin{equation} \label{eq:linill0}
\|(A+\alpha I)^{-1}\| \le \frac{1}{\alpha}
\end{equation}
are valid (cf.~\cite[Section~7.1.1]{Haase06}. Regularized solutions $\xad$ of Lavrentiev regularization in the linear case attain  the explicit form
\begin{equation} \label{eq:Lavex}
\xad=(A+\alpha I)^{-1}(\yd + \alpha\,\bar x),
\end{equation}
because they solve the equation
\begin{equation} \label{eq:Lavlin}
A\xad+\alpha(\xad-\bar{x})=\yd.
\end{equation}

Since the properties of a linear operator $A \in \mathcal{L}(X)$ do not depend on the solution point $\xdag$, well-posedness and ill-posedness of the operator equation (\ref{eq:linopeq}) in the sense of Definition~\ref{def:posed} are global properties. Thus, the equation is locally well-posed everywhere or locally ill-posed everywhere as the following proposition outlines.

\begin{proposition} \label{pro:distinguish}
The linear operator equation (\ref{eq:linopeq}) is under (\ref{eq:Amonotone}) locally well-posed everywhere if and only if $A$ is continuously invertible, i.e.~if $A^{-1} \in \mathcal{L}(X)$ and we have a constant $K>0$
such that
\begin{equation} \label{eq:linwell}
\|(A+\alpha I)^{-1}\| \le K<\infty   \qquad \mbox{for all} \quad \alpha>0,
\end{equation}
where $K=\|A^{-1}\|$ holds true. Alternatively, (\ref{eq:Amonotone}) is locally ill-posed everywhere if and only if the nullspace of $A$ is non-trivial, i.e.~$\mathcal{N}(A)\ne\{0\}$, or the range $R(A)$ of $A$ is not closed.
Then we have
\begin{equation} \label{eq:linill}
\|(A+\alpha I)^{-1}\| = \frac{1}{\alpha}   \qquad \mbox{for all} \quad \alpha>0.
\end{equation}
\end{proposition}
\begin{proof}
The well-posed case (\ref{eq:linwell}) is characterized by $0 \notin \sigma(A)$, where $\sigma(A)$ denotes the spectrum of the operator $A$, whereas the ill-posed case is characterized by $0 \in \sigma(A)$.

For $0 \notin \sigma(A)$ we have by definition $A^{-1}\in \mathcal{L}(X)$, i.e.~with bijective operator $A:X \to X$, $\|A^{-1}\|<\infty$ and $\|A^{-1}(y-\tilde y) \|\le \|A^{-1}\|\,\|y-\tilde y\|$ for all $y,\tilde y \in X$, which indicates the local well-posedness everywhere. With $A$ also $A^{-1}$ is monotone and thus we can estimate with $K:=\|A^{-1}\|$ as
$$\|(A+\alpha I)^{-1}\|=\|A^{-1}(I+\alpha A^{-1})^{-1}\|\le \|A^{-1}\|\,\|(I+\alpha A^{-1})^{-1}\|\le K \quad \mbox{for all} \quad \alpha>0.$$

For $0 \in \sigma(A)$ we have by definition that at least one of the assertions $N(A) \neq \{0\}$ and $R(A)$ is not closed is true,
which can be summarized by the condition $\mathcal{R}(A) \ne X$ due to the orthogonal sum
\begin{equation} \label{eq:plussum}
\overline{\mathcal{R}(A)} \oplus \mathcal{N}(A)=X
\end{equation}
(cf., e.g., \cite[Theorem 1.1.10]{Plato95}). Both cases indicate local ill-posedness everywhere. This is obvious for  $\mathcal{N}(A)\ne\{0\}$. In the case $\mathcal{N}(A)=\{0\}$, but
$\mathcal{R}(A)\ne \overline{\mathcal{R}(A)}$, we have that $A^{-1}$ exists and is an unbounded linear operator and hence for all $r>0$ that there is sequence $\{x_n\} \subset X$ with $\|x_n\|=r$ and $\lim \limits_{n \to \infty}\|Ax_n\|=0$. Then, we have for  $\xdag+x_n \in \mathcal{B}_r(\xdag)$ the limit properties $\xdag+x_n \not \to \xdag$ but $A(\xdag+x_n) \to A \xdag$ as $n \to \infty$ and thus local ill-posedness
at $\xdag$.
Taking into account (\ref{eq:linill0}), to prove (\ref{eq:linill}) it remains to show $\|(A+\alpha I)^{-1}\| \ge \frac{1}{\alpha}$ for all $\alpha>0$. This, however, is a consequence of the Neumann series theory, which
says that we have, for a bounded linear operator $B: X \to X$ with $\|B\|<1$, that $(I-B)^{-1} \in \mathcal{L}(X)$. By setting $B:=(\frac{1}{\alpha}A+I)^{-1}$, we must have  $\|B\| \ge 1$. Otherwise, we would get that $0 \notin \sigma(A)$. Evidently, the conditions (\ref{eq:linwell}) and (\ref{eq:linill}) are incompatible, but one of them is always true for a bounded monotone operator $A$. Now the proof is complete.
\end{proof}

Note that the ill-posed case in Proposition~\ref{pro:distinguish} with $\mathcal{R}(A)\ne \overline{\mathcal{R}(A)}$ can only occur if $X$ is an infinite dimensional space and the range $\mathcal{R}(A)$ is also infinite dimensional. Moreover, it will be a by-product of the assertion of Proposition~\ref{pro:pro2} (cf.~(\ref{eq:kleinO})) below that for arbitrary monotone operators $A \in \mathcal{L}(X)$ the condition
$$\|(A+\alpha I)^{-1}x\|=O(1) \qquad \mbox{as} \qquad  \alpha \to 0, $$
which is in the case (\ref{eq:linwell}) valid for all $x \in X$,
cannot be improved to
$$\|(A+\alpha I)^{-1}x\|=o(1) \qquad \mbox{as} \qquad  \alpha \to 0 $$
if $x \ne 0$.

It is evident that strong monotonicity \begin{equation} \label{eq:Amonstrong}
 \langle Ax,x \rangle \geq C\,\|x\|^2 \qquad \mbox{for all} \quad x \in X,
\end{equation}
with some constant $C>0$ implies $\|x\| \le \frac{1}{C}\|Ax\|$ for all $x \in X$ and hence with $A^{-1}\in \mathcal{L}(X)$ local well-posedness of (\ref{eq:Amonotone}) everywhere. Vice versa,
$A^{-1}\in \mathcal{L}(X)$ does not, in general, imply strong monotonicity, because we have $\langle Ax,x \rangle=0$ for all $x \in X$ if the monotone operator $A$ is skew-symmetric, i.e.~for the adjoint operator $A^*$ that $A^*=-A$. The simplest case of such behaviour is $A=\left( \begin{array}{cc} 0 & -1 \\ 1 & 0 \end{array}  \right)$ for $X=\mathbb{R}^2$.

\begin{remark} \label{rem:ill}
{\rm For linear equations (\ref{eq:linopeq}) with monotone $A \in \mathcal{L}(X)$, the case distinction (cf.~Proposition~\ref{pro:distinguish}) between locally well-posed and ill-posed situations based on Definition~\ref{def:posed} is different from the usual
case distinction in the literature of linear regularization theory (see, e.g., \cite{Nashed87}), where a bounded pseudoinverse $A^\dagger$ characterized by $\mathcal{R}(A)=\overline{\mathcal{R}(A)}$ denotes well-posedness and an unbounded $A^\dagger$ characterized by $\mathcal{R}(A)\ne \overline{\mathcal{R}(A)}$  denotes ill-posedness. However, we will see below in Proposition~\ref{pro:differ} that the concept of Definition~\ref{def:posed} is the more appropriate one for our setting in the context of Lavrentiev regularization.
}\end{remark}

The remaining part of the paper is organized as follows: In Section~\ref{s2}, we discuss the impact of conditional stability on error estimates and convergence rates for the Lavrentiev regularization. Furthermore, we mention in Proposition~\ref{pro:differ}
some saturation result from \cite{Plato16} for the linear case which motivates to distinguish well-posedness and ill-posedness on the basis of Definition~\ref{def:posed}. The role of the regularization error in the noise-free case, called
bias, will be investigated in Section~\ref{s3} for nonlinear and linear problems. For linear operator equations general convergence rates, including logarithmic rates, are derived in Section~\ref{s4} by means of the method of approximate source conditions. This allows us to extend well-known convergence rates results for the Lavrentiev regularization, which were based on general source conditions, to the case of non-selfadjoint linear monotone forward operators for which general source conditions fail.
Examples presenting the self-adjoint multiplication operator as well as the non-selfadjoint fractional integral operator and Ces\`aro operator illustrate the theoretical results of this section. Extensions to the nonlinear case under specific
conditions on the nonlinearity structure in Section~\ref{s5} complete the paper.

\section{Error estimates and the case of conditional stability}\label{s2}
\setcounter{equation}{0}
\setcounter{theorem}{0}

\begin{proposition} \label{pro:exist}
Let under the Assumption~\ref{ass:ass0} the solution set to equation (\ref{eq:nonlinopeq})
$$L:=\{x \in X:\;F(x)=y\}$$
be nonempty. This set $L$ is closed and convex, and consequently there is a uniquely determined $\bar x$-minimum norm solution $\xdag_{mn} \in L$ to (\ref{eq:nonlinopeq}) such that $$\|\xdag_{mn}-\bar x\|= \min \{\|\xdag-\bar x\|:\;\xdag \in L\}. $$
The Lavrentiev-regularized solution $\xad \in X$ is uniquely determined, which means that (\ref{eq:lav}) has a unique solution $\xad$ for all $\bar{x} \in X$,
$\yd \in X$ and $\alpha>0$, which depends continuously on $\yd$. Moreover, for any solution $\xdag \in L$, the following three basic inequalities are valid:
\begin{eqnarray}
&&\|\xad-\xdag\|^2\leq \langle \xdag-\bar{x},\xdag-\xad\rangle+\frac{\delta}{\alpha}\|\xad-\xdag\|,\label{eq:IE1}\\
&&\|\xad-\xdag\|\leq \|\xdag-\bar{x}\|+\frac{\delta}{\alpha},\label{eq:IE2}\\
&&\|F(\xad)-F(\xdag)\|\leq \alpha\|\xdag-\bar{x}\|+\delta.\label{eq:IE3}
\end{eqnarray}
\end{proposition}
\begin{proof}
The closedness and convexity of $L$ is due to the maximal monotonicity of $F: X \to X$ (cf.~\cite[Prop.~23.39]{Bauschke11}). Then the $\bar x$-minimum norm solution $\xdag_{mn}$ is the uniquely determined
best approximation of $\bar x$ in $L$.  The next assertion of the proposition is a consequence of the Browder-Minty theorem which ensures that for all $\alpha>0$ the operator $F+\alpha I:X \to X$ is bijective and strongly monotone such that $\xad$ is uniquely determined and depends continuously on the data $\yd$. As outlined in \cite{HKR16},
by testing (\ref{eq:lav}) with the two elements $\xad-\xdag$ and $F(\xad)-F(\xdag)$ we obtain
\begin{eqnarray} \label{eq:F1}
&&\langle F(\xad)-F(\xdag), \xad-\xdag\rangle + \langle y-\yd,\xad-\xdag\rangle \nonumber
\\&&+\alpha \|\xad-\xdag\|^2+\alpha \langle \xdag-\bar{x},\xad-\xdag\rangle=0
\end{eqnarray}
and
\begin{eqnarray} \label{eq:F2}
&&\|F(\xad)-F(\xdag)\|^2 + \langle y-\yd,F(\xad)-F(\xdag)\rangle \nonumber
\\&&+\alpha \langle F(\xad)-F(\xdag), \xad-\xdag\rangle+\alpha \langle \xdag-\bar{x},F(\xad)-F(\xdag)\rangle=0\,,
\end{eqnarray}
respectively. By using the monotonicity  (\ref{eq:Fmon}) of $F$, the Cauchy-Schwarz inequality yields (\ref{eq:IE1}) and moreover (\ref{eq:IE2}), as a consequence of (\ref{eq:F1}), while (\ref{eq:IE3}) follows as a consequence of (\ref{eq:F2}). This completes the proof.
\end{proof}

\begin{proposition} \label{pro:errorcondstab}
For a solution $\xdag$ of equation (\ref{eq:nonlinopeq}), assume that a conditional stability estimate of the form (\ref{eq:condstab}) with $\mathcal{D}(F)=X,\;Q=\mathcal{B}_r(\xdag)$, and some concave index function $\varphi$ holds.
Then the solution set $L$ is a singleton, i.e.~$L=\{\xdag\}$, and for an a priori parameter choice $\alpha=\alpha(\delta)=c\,\delta,\;c>0,$  the Lavrentiev regularized solutions $\xad$ convergence to $\xdag$ with the rate
\begin{equation} \label{eq:rate1}
\|x_{\alpha(\delta)}^\delta-\xdag\| =\mathcal{O}(\varphi(\delta)) \qquad \mbox{as} \qquad \delta \to 0
\end{equation}
if the radius $r$ in (\ref{eq:condstab}) is sufficiently large such that $r > \|\xdag-\bar{x}\|+\frac{1}{c}$.
\end{proposition}
\begin{proof}
A conditional stability estimate of the form (\ref{eq:condstab}) with $\mathcal{D}(F)=X,\;Q=\mathcal{B}_r(\xdag)$ and some $r>0$ ensures that $L \cap \mathcal{B}_r(\xdag)=\{\xdag\}$ is a singleton. As $L$ is a convex set (cf.~Proposition~\ref{pro:exist}),
there cannot be a second element in $L$. Now we have from (\ref{eq:IE2}) that $\|x_{\alpha(\delta)}^\delta-\xdag\|\leq \|\xdag-\bar{x}\|+\frac{1}{c}$ and thus $x_{\alpha(\delta)}^\delta \in \mathcal{B}_r(\xdag)$. Then (\ref{eq:condstab}) and (\ref{eq:IE3}) yield $$\|x_{\alpha(\delta)}^\delta-\xdag\|\leq \varphi(\|F(x_{\alpha(\delta)}^\delta)-F(\xdag)\|) \leq \varphi(c\,\delta\,\|\xdag-\bar{x}\|+\delta) \leq 2\,\max(1,c\,\|\xdag-\bar{x}\|)\,\varphi(\delta).$$
This proves the proposition.
\end{proof}

As was mentioned above for the Tikhonov regularization, Proposition~\ref{pro:errorcondstab} shows that also the Lavrentiev regularization ensures convergence and rates of the approximate solutions
under the conditional stability estimate (\ref{eq:condstab}) by embedding the regularized solutions in the stability region, which is here $\mathcal{B}_r(\xdag)$.
From Propositions~\ref{pro:condstab} and \ref{pro:errorcondstab}  we immediately arrive at the following corollary.

\begin{corollary}\label{cor:condstab} Choose the regularization parameter for the Lavrentiev regularization a priori as $\alpha(\delta) \sim \delta$.
If $F$ is strongly monotone with sufficiently large $r>0$ in (\ref{eq:Fmonstrongloc}), then we have a linear (Lipschitz) convergence rate
\begin{equation} \label{eq:linconrate}
\|x_{\alpha(\delta)}^\delta-\xdag\| =\mathcal{O}(\delta) \qquad \mbox{as} \qquad \delta \to 0.
\end{equation}
If $F$ is uniformly monotone with $\zeta(t)=t^{\kappa+1}, \;\kappa>1$, and sufficiently large $r>0$ in (\ref{eq:Fmongenloc}), then we have a H\"older convergence rate
\begin{equation} \label{eq:Hrate}
\|x_{\alpha(\delta)}^\delta-\xdag\| =\mathcal{O}(\delta^{1/\kappa}) \qquad \mbox{as} \qquad \delta \to 0.
\end{equation}
\end{corollary}

\begin{example}\label{ex:power}
{\rm (One dimensional example) For $X:=\mathbb{R}$ with $\|x\|:=|x|$ we consider the continuous monotone operator $F:\mathbb{R} \to \mathbb{R}$ defined for exponents $\kappa>0$ as
$$F(x):=\left\{\begin{array}{ccc} -1 & \mbox{if} & -\infty < x < -1 \\ -(-x)^\kappa & \mbox{if} & -1 \le x  \le 0 \\x^\kappa & \mbox{if} & 0 < x  \le 1 \\ 1 & \mbox{if} & 1 < x < \infty \end{array} \right.,$$
which however is not bijective and not coercive. Then we have obviously local ill-posedness at $\xdag$ if $\xdag<-1$ or $\xdag>1$. On the other hand we have for all $\kappa>0$  local well-posedness at $\xdag=0$, because the
local uniform monotonicity condition  (\ref{eq:Fmongenloc}) is satisfied there with $\zeta(t)=t^{\kappa+1}$ such that Proposition~\ref{pro:condstab} and Corollary~\ref{cor:condstab} apply for $\xdag=0$ with $\theta(t)=t^\kappa$ for $\kappa \ge 1$. Indeed, a superlinear convergence rate (\ref{eq:Hrate}) at $\xdag=0$ occurs if $0<\kappa<1$.
}\end{example}

For the special case of monotone linear operators $A \in \mathcal{L}(X)$ we have the two different situations formulated in Proposition~\ref{pro:differ}. This indicates a significant gap in the convergence rates and motivates the specific case distinction
between well-posedness and ill-posedness based on Definition~\ref{def:posed} also for linear monotone operators $A$ as mentioned above in Remark~\ref{rem:ill}.

\begin{proposition} \label{pro:differ}
For the maximal best possible error
$$E_{\xdag}(\delta):=\sup \limits_{\yd \in X:\,\|y-\yd\| \le \delta} \,\inf \limits_{\alpha>0} \|\xad-\xdag\| $$
of Lavrentiev regularization to equation (\ref{eq:linopeq}) with bounded monotone linear operator $A$ we have on the one hand
\begin{equation} \label{eq:satwell}
E_{\xdag}(\delta)=\mathcal{O}(\delta) \qquad \mbox{as} \qquad \delta \to 0
\end{equation}
for all $\xdag \in X$ if (\ref{eq:linwell}) is valid, i.e.~if $A$ is continuously invertible. On the other hand, we have that
\begin{equation} \label{eq:satres}
E_{\xdag}(\delta)=o(\sqrt{\delta}) \qquad \mbox{as} \qquad \delta \to 0 \qquad \mbox{implies} \qquad \xdag-\bar x=0
\end{equation}
if (\ref{eq:linwell}) is violated for arbitrarily large $K>0$, that is exactly the case if the null-space $\mathcal{N}(A)$ of $A$ is not trivial or the range $\mathcal{R}(A)$ of $A$ is not closed.
\end{proposition}
\begin{proof}
For the special case of monotone linear operators $A \in \mathcal{L}(X)$ we find directly from (\ref{eq:Lavex})
the error estimate
\begin{equation} \label{eq:linformula}
\|\xad-\xdag\| \le \|(A+\alpha I)^{-1}[(\yd-y)+\alpha(\bar x-\xdag)]\| \le \|(A+\alpha I)^{-1}\|(\delta+\alpha\,\|\xdag-\bar x\|).
\end{equation}
This ensures the linear convergence rate (\ref{eq:linconrate}) for the regularization parameter choice $\alpha(\delta) \sim \delta$ in the well-posed case (\ref{eq:linwell}) and hence (\ref{eq:satwell}).
The implication (\ref{eq:satres}), however, recalls the recently published saturation result from Theorem 5.1 in \cite{Plato16} for the ill-posed case. This proves a significant gap in the convergence rates between well-posed
and ill-posed situations.
\end{proof}

It should be mentioned that the saturation result (\ref{eq:satres}) for noisy data is a Lavrentiev regularization analogue to the well-known saturation result \cite[Proposition 5.3]{EHN96} for the Tikhonov regularization.

\section{The distinguished role of bias}\label{s3}
\setcounter{equation}{0}
\setcounter{theorem}{0}

For the error analysis of Lavrentiev-regularized solutions it is helpful to consider in addition to $\xad$  the regularized solutions $\xa:=x^0_\alpha$ in the noise-free case $(\delta=0)$, which satisfy the operator equations
\begin{equation} \label{eq:Lavnonlin0}
F(\xa)+\alpha(\xa-\bar{x})=y
\end{equation}
and
\begin{equation} \label{eq:Lavlin0}
A\xa+\alpha(\xa-\bar{x})=y
\end{equation}
in the nonlinear and linear case, respectively. From Proposition~\ref{pro:exist} we have that also the elements $\xa \in X$ are uniquely determined for all $\alpha>0$. It is obvious in regularization theory that the total norm error of regularization can be estimated above by the triangle inequality as
\begin{equation} \label{eq:triangle}
\|\xad-\xdag\| \le \|\xa-\xdag\|+\|\xa-\xad\|
\end{equation}
such that an upper bound of the noise propagation error $\|\xa-\xad\|$ is independent of the solution $\xdag$, but depends on the noise level $\delta$ and on the regularization parameter $\alpha>0$.
If the regularization procedure $R_\alpha$ is expressed by a continuous linear mapping $\yd \mapsto \xad$, then estimates of the form   $\|\xa-\xad\|\le \|R_\alpha\|\delta$ are standard, where
$\lim_{\alpha \to 0}\|R_\alpha\|=\infty$ takes place for the ill-posed case. So we have for (\ref{eq:Lavlin0}), in the case $\bar x=0$, $R_\alpha y=(A+\alpha I)^{-1}y$ with
\begin{equation} \label{eq:noiseprop}
\|\xa-\xad\|\le \frac{\delta}{\alpha}
\end{equation}
for all $\delta \ge 0$ and $\alpha>0$ due to (\ref{eq:linill0}). For nonlinear ill-posed problems, however, the regularization procedure is in general characterized by a nonlinear mapping $\yd \mapsto \xad$.
Estimates of $\|\xa-\xad\|$ from above independent of $\xdag$ are then restricted to classes of forward operator $F$ with specific nonlinearity properties, and we refer for example to the discussion in \cite{SEK93}
for estimates of the form $\|\xa-\xad\|\le \frac{c\delta}{\sqrt{\alpha}}$ versus  $\|\xa-\xad\|\le \frac{c\delta}{\alpha}$ for the nonlinear Tikhonov regularization (cf.~(\ref{eq:Tik})). Taking advantage of the monotonicity of $F$ the situation is simpler for the nonlinear Lavrentiev regularization as the following lemma shows.

\begin{lemma} \label{lem:lem1}
For arbitrary monotone forward operators $F$ and solutions $\xdag \in X$ we have the uniform propagation error estimate (\ref{eq:noiseprop}) for the Lavrentiev regularization with $\xad$ and $\xa$ from
(\ref{eq:lav}) and (\ref{eq:Lavnonlin0}), respectively.
\end{lemma}
\begin{proof}
For $\xad-\xa = 0$, (\ref{eq:noiseprop}) is trivially satisfied.  As  difference of the two equations (\ref{eq:lav}) and (\ref{eq:Lavnonlin0}) we have the equation
$$F(\xad)-F(\xa)+\alpha(\xad-\xa)=\yd-y$$ and thus by testing with $\xad-\xa \ne 0$
$$\langle F(\xad)-F(\xa),\xad-\xa \rangle+\alpha\,\|\xad-\xa\|^2=\langle \yd-y, \xad-\xa\rangle \le \|\xad-\xa\|\, \delta.$$
Due to the monotonicity property (\ref{eq:Fmon}) this implies the inequality
$\alpha\,\|\xad-\xa\| \le \delta$ and hence  (\ref{eq:noiseprop}) which completes the proof.
\end{proof}

Taking into account the bound (\ref{eq:noiseprop}) of the noise propagation error,
for fixed forward operator $F$ and fixed solution $\xdag$ the asymptotics of the total regularization error $\|\xad-x^\dagger\|$ as $\delta \to 0$ of Lavrentiev regularization is essentially influenced by the asymptotics of the bias (regularization error for noise-free data) $$B^F_{\xdag}(\alpha):=\|\xa-\xdag\|$$
as $\alpha \to 0$. From \cite[Section~23]{Bauschke11} we derive the following proposition for the bias of the Lavrentiev regularization.

\begin{proposition} \label{pro:pro10}
Under Assumption~\ref{ass:ass0} let $\xdag$ solve the equation (\ref{eq:nonlinopeq}). Then we have
\begin{equation} \label{eq:limitbias0}
\lim \limits _{\alpha \to 0} B_{\xdag}^F(\alpha) =0
\end{equation}
if and only if  $\xdag=\xdag_{mn} \in X$, i.e.~$\xdag$ is the $\bar x$-minimum norm solution to equation (\ref{eq:nonlinopeq}).
\end{proposition}
\begin{proof}
Under the stated assumptions we have from Proposition \ref{pro:exist} that the nonempty set $L:=\{x \in X: F(x)=y\}$  is closed and convex and hence the projection of $\bar x$ onto this set is an $\bar x$-minimum norm solution to equation (\ref{eq:nonlinopeq}) and uniquely determined. Then we have
from Theorem 23.44 (i) in \cite{Bauschke11} that the uniquely determined element $\xa \in X$ existing for all $\alpha>0$ and $\bar x \in X$, which satisfies the equation
$$F(\xa)-y+\alpha(\xa-\bar x)=0,$$
tends in the norm of $X$ to the projection of $\bar x$ on $L$. If $\xdag$ solves the equation (\ref{eq:nonlinopeq}), but fails to be an $\bar x$-minimum norm solution, then the projection of $\bar x$ on $L$ differs from $\xdag$ and
(\ref{eq:limitbias0}) cannot hold. This proves the proposition.
\end{proof}

The asymptotic behaviour of the bias $B^F_{\xdag}(\alpha) \to 0$ as $\alpha \to 0$ expressing the intrinsic smoothness of the solution $\xdag$ with respect to the forward operator $F$ (cf.~in a more general context the ideas in \cite[Chapt.~3]{Schuster12} and \cite{Hof15}) fully determines the specific error profile for the solution $\xdag$. Therefore the bias was called `profile function' in the former paper \cite{HofMat07} with focus on a general regularization scheme for linear ill-posed problems.

\begin{remark} \label{rem:rem1}
{\rm For the Lavrentiev regularization (\ref{eq:Lavlin}) to linear problems (\ref{eq:linopeq}) with monotone forward operator $A \in \mathcal{L}(X)$, Proposition~\ref{pro:pro10} applies
and based on formula (\ref{eq:plussum}) we note that a solution $\xdag$ to the equation (\ref{eq:linopeq}) is an $\bar x$-minimum solution $\xdag_{mn}$ if and only if $\xdag - \bar x$ is orthogonal to the null-space of the linear operator $A$, i.e.
\begin{equation} \label{eq:orthonull}
\xdag - \bar x \perp \mathcal{N}(A).
\end{equation}
The corresponding bias attains the form
 \begin{equation} \label{eq:biaslin}
B^A_{\xdag}(\alpha)=\|\xa-\xdag\|=\alpha\,\|(A+\alpha I)^{-1}(\xdag-\bar x)\|
\end{equation}
and we have $\lim \limits_{\alpha \to 0} B^A_{\xdag_{mn}}(\alpha)=0$, but $B^A_{\xdag}(\alpha) \not \to 0$ as $\alpha \to 0$ when the solution $\xdag$ to (\ref{eq:linopeq}) fails to satisfy (\ref{eq:orthonull}).
}\end{remark}

The following considerations are only of interest for the ill-posed case, because the estimate (\ref{eq:totallin}) below is not helpful for the well-posed case, in which (\ref{eq:linformula}) directly yields the linear rate (\ref{eq:linconrate}) for all $\xdag \in X$.
In the ill-posed case, however, the asymptotics of $B^A_{\xdag}(\alpha)$ for $\alpha \to 0$ determines, for example by  equilibrating the two terms in the right hand side of the inequality
 \begin{equation} \label{eq:totallin}
\|\xad-\xdag\| \le \alpha\,\|(A+\alpha I)^{-1}(\xdag-\bar x)\| + \frac{\delta}{\alpha},
\end{equation}
the chances and limitations of possible convergence rates of the total regularization error. This point was intensively analyzed in \cite{Taut02} for fractional power source conditions
 \begin{equation} \label{eq:directp}
\xdag-\bar x= A^p\,w, \qquad w \in X,\quad\;0<p\le 1,
\end{equation}
yielding for all $0<p \le 1$ the H\"older type convergence rates of the bias
\begin{equation} \label{eq:ratep}
B^A_{\xdag}(\alpha) =O(\alpha^p) \qquad \mbox{as}\quad \alpha \to 0.
\end{equation}
We note that for monotone operators $A \in \mathcal{L}(X)$ the fractional powers $A^p,\;0<p \le 1,$ are defined via the Balakrishnan calculus as
\begin{equation} \label{eq:Bala}
A^p := \frac{\sin \pi p }{\pi}
\int_0^\infty s^{p-1} \left(A + sI\right)^{-1} A \  ds.
\end{equation}
Because of the saturation result for the bias presented with the following proposition, we call the source condition
\begin{equation} \label{eq:direct1}
\xdag-\bar x= A\,w, \quad w \in X,
\end{equation}
benchmark source condition.
\begin{proposition} \label{pro:pro2}
With the exception of the singular case $\xdag-\bar{x}=0$ the benchmark source condition (\ref{eq:direct1}) yields with
\begin{equation} \label{eq:bestO}
B^A_{\xdag}(\alpha) =O(\alpha) \qquad \mbox{as}\quad \alpha \to 0
\end{equation}
the best possible bias rate, because
\begin{equation} \label{eq:kleinO}
B^A_{\xdag}(\alpha) =o(\alpha) \qquad \mbox{as}\quad \alpha \to 0 \qquad \mbox{implies} \qquad \xdag-\bar{x}=0.
\end{equation}
\end{proposition}
\begin{proof}
Under the benchmark source condition we have $B^A_{\xdag}(\alpha) =\alpha\|(A+\alpha I)^{-1}Aw\| \le \alpha \|w\|=O(\alpha)$ as $\alpha \to 0$ due to (\ref{eq:monprop}). To prove the implication (\ref{eq:kleinO}) we distinguish the cases
$A=0$ and $A \ne 0$. For $A=0$ we have $B^0_{\xdag}(\alpha)=\|\xdag-\bar x\|$, and the implication (\ref{eq:kleinO}) is evidently true. In the case $A \ne 0$ we conclude as follows:
For all $x \in X$ it holds $\|(A+\alpha I)x\| \le \|Ax\|+\alpha\|x\| \le (\|A\|+\alpha)\|x\|$. Moreover, we have
$\|\xdag-\bar{x}\| \le (\|A\|+\alpha)\|(A+\alpha I)^{-1}(\xdag-\bar{x})\|$ for arbitrary $\xdag-\bar{x} \in X$, which is for all $\alpha>0$ equivalent to
\begin{equation} \label{eq:proof1}
\frac{\|\xdag-\bar{x}\|}{\|A\|+\alpha}\; \le\; \|(A+\alpha I)^{-1}(\xdag-\bar{x})\|.
\end{equation}
On the other hand, from (\ref{eq:proof1}) we derive for $A \ne 0$ the inequality
$$\liminf \limits_{\alpha \to 0}\|(A+\alpha I)^{-1}(\xdag-\bar{x})\|\; \ge \; \frac{\|\xdag-\bar{x}\|}{\|A\|},  $$
which for $\xdag -\bar{x}\ne 0$ and $\frac{\|\xdag-\bar{x}\|}{\|A\|}>0$ violates the limit condition $$\lim \limits _{\alpha \to 0} \frac{B^A_{\xdag}(\alpha)}{\alpha}= \lim \limits _{\alpha \to 0}\|(A+\alpha I)^{-1}(\xdag-\bar{x})\|=0.$$
This completes the proof.
\end{proof}

As already mentioned, the rates of the bias $B^A_{\xdag}(\alpha)$ for $\alpha \to 0$ also determine the total error profile on the basis of the estimate (\ref{eq:totallin}). In the simplest case of an
a priori choice $\alpha=\alpha(\delta) \sim \delta^\frac{1}{p+1}$ we directly derive for all $0<p \le 1$ H\"older convergence rates
\begin{equation} \label{eq:linHoelder}
\|x_{\alpha(\delta)}^\delta-\xdag\| =O(\delta^\frac{p}{p+1}) \qquad \mbox{as}\quad \delta \to 0
\end{equation}
from the source conditions (\ref{eq:directp}).
In the benchmark case $p=1$ this gives
\begin{equation} \label{eq:best}
\|x_{\alpha(\delta)}^\delta-\xdag\| =O(\delta^\frac{1}{2}) \qquad \mbox{as}\quad \delta \to 0,
\end{equation}
and Plato's saturation theorem \cite[Theorem 5.1]{Plato16} (cf.~formula (\ref{eq:satres}) of Proposition~\ref{pro:differ}) proves that for the ill-posed situation this is the maximal best possible error rate for the Lavrentiev regularization in the linear case, since $\|x_{\alpha(\delta,\yd)}^\delta-\xdag\| =o(\delta^\frac{1}{2})$
implies for arbitrary a posteriori choices $\alpha=\alpha(\delta,\yd)$ of the regularization parameter that $\xdag-\bar x=0$.

Recently, it was shown in \cite{PlMaHo16} that alternative source conditions
\begin{equation*} \label{eq:adjointp}
\xdag-\bar x= (A^*)^p\,w, \quad w \in X,
\end{equation*}
which replace the monotone non-selfadjoint operator $A$ with the also monotone adjoint $A^*$, can be less efficient with respect to rate results if $1/2 \le p \le 1$. In the worst case, the best possible bias rate under the adjoint source condition $\xdag-\bar{x}=A^*w$ is \linebreak $B^A_{\xdag}(\alpha) =O(\sqrt{\alpha})$.
This worst case, for example, takes place
with $X:=L^2(0,1)$ when we consider the Riemann-Liouville fractional integral operator $A:=V$ studied below in Example~\ref{ex:V}.
 This case is connected with a reduced total error rate $\|\xad-\xdag\| =O(\delta^\frac{1}{3})$ for $p=1$ in comparison to (\ref{eq:best}). Consequently, the situation of Lavrentiev regularization differs significantly from the situation of Tikhonov's regularization method, where just this adjoint source condition is advantageous (cf.~\cite[Corollary 3.1.3]{Groe84}).

In Section~\ref{s4}, by exploiting the above mentioned bias studies and by using the method of approximate source conditions with benchmark condition  (\ref{eq:direct1}),
we will extend the results to general, non-H\"older type, and low order convergence rates occurring in the
context of linear Lavrentiev regularization. We note that the focus is on non-selfadjoint operators $A$, where spectral theory fails.
Since the solution-independent bound (\ref{eq:noiseprop}) for the noise propagation error is also valid for the Lavrentiev regularization (\ref{eq:lav}) applied to nonlinear equations
(\ref{eq:nonlinopeq}), we will show  in Section~\ref{s5} that a bias-based error analysis can also be successful for classes of monotone forward operators $F$ under specific restrictions of the nonlinearity structure.

\section{General convergence rates for the linear case using approximate source conditions}\label{s4}
\setcounter{equation}{0}
\setcounter{theorem}{0}

Based on the paper \cite{MatHof08}, for a given selfadjoint non-negative linear operator $H \in \mathcal{L}(X)$ with non-closed range $\mathcal{R}(H) \not=\overline{\mathcal{R}(H)}$ in the Hilbert space $X$,
it can be shown that for every element $u \in X$ with $u \perp \mathcal{N}(H)$  there exist an index function $\varphi$ and a source element $w \in X$ such that a general source condition
$u=\varphi(H)w$ holds, where $\varphi(H)$ as usual is defined by spectral theory such that any spectrum point $\lambda>0$ of $H$ corresponds to the spectrum point $\varphi(\lambda)$ of $\varphi(H)$. For a selfadjoint monotone operator $A \in \mathcal{L}(X)$ with non-closed range this proves with $H:=A$ under (\ref{eq:orthonull}) a source condition
\begin{equation} \label{eq:gensource}
\xdag-\bar x=\varphi(A)w,\quad w\in X.
\end{equation}
Using spectral properties of $A$ this makes it possible to formulate convergence rates $B^A_{\xdag}(\alpha)=O(\psi(\alpha))$ as $\alpha \to 0$ for the Lavrentiev regularization bias with some index function $\psi$ which depends on the index function $\varphi$, as is similarly done for the Tikhonov regularization bias with $H:=A^*A$ (cf., e.g.,~ \cite{AndreevEtal15,FHM11,MatHof08}). Some authors exploit this approach for the Lavrentiev regularization, partially even in a nonlinear setting, see \cite{Argyros14,MaNa13,NairTau04,Seme10},
but their restriction to selfadjoint monotone operators $A \in \mathcal{L}(X)$ is rather artificial, because in particular the case of non-selfadjoint monotone linear operators (see Examples~\ref{ex:V} and \ref{ex:C} below)
is of interest. For such operators $A$, however, spectral theory is not applicable and the Balakrishnan calculus (cf.~(\ref{eq:Bala})) only allows us to handle power type source conditions (\ref{eq:directp}) yielding H\"older convergence rates (\ref{eq:ratep}) for the bias $B^A_{\xdag}(\alpha)$ as $\alpha \to 0$ and consequently yielding only H\"older convergence rates
$\|\xad-\xdag\| =O(\delta^\frac{p}{p+1})$ as $\delta \to 0$ for $0<p \le 1$ and noisy data when taking into account the estimates (\ref{eq:triangle}) and (\ref{eq:noiseprop}).

As already mentioned in \cite[Section~4.1]{HKR16}, by avoiding expressions $\varphi(A)$ with index functions $\varphi$ of non-power type and non-selfadjoint monotone linear operators $A$, the method of approximate source conditions can help to verify low order convergence rates of non-H\"older type for the Lavrentiev regularization without self-adjontness
assumptions of the forward operator in the linear case or of its Fr\'echet derivative at the solution in the nonlinear case. We will outline details of such an approach
for the linear case in this section and for the nonlinear case in the subsequent section. The method of approximate source conditions had been developed for linear ill-posed operator equations in Hilbert spaces in \cite{Hof06} (see also \cite{DHY07}) and was extended to nonlinear equations and a Banach space setting in \cite{HeinHof09} and \cite{BoHo10} (see also \cite{Schuster12}).
Associated with the best possible rate (\ref{eq:bestO}) the condition (\ref{eq:direct1}) acts in an optimal manner as benchmark source condition for obtaining convergence rates if $\xdag-\bar x$ satisfies
(\ref{eq:orthonull}), but violates (\ref{eq:direct1}). In such case the smoothness of the element $\xdag-\bar x$ with respect to the monotone operator $A$ is too small for having the bias rate (\ref{eq:bestO}) and
one can use the distance function
\begin{equation} \label{eq:distfunct}
d(R)=\min\{\|\xdag-\bar{x}-Aw\| : \, \|w\|\leq R\}, \qquad 0 \le R < \infty,
\end{equation}
to measure  for $\xdag-\bar{x}$ the degree of violation with respect to the benchmark condition expressed by the decay rate of $d(R) \to 0$ as $R \to \infty$.

\begin{lemma} \label{lem:lem2}
Assume that for the monotone operator $A \in \mathcal{L}(X)$ the element $\xdag \in X$ fails the benchmark source condition  (\ref{eq:direct1}), i.e.~$\xdag-\bar{x} \notin \mathcal{R}(A)$, but satisfies the
orthogonality condition (\ref{eq:orthonull}). Then the distance function $d(R)$ from (\ref{eq:distfunct}) is positive, strictly decreasing and concave and hence continuous for all $0 \le R<\infty$ and satisfies the limit
condition $\lim\limits_{R \to \infty} d(R)=0$.
\end{lemma}
\begin{proof}
The assertion of the lemma follows immediately from \cite[Lemma~3.2]{BoHo10} if the condition $\xdag - \bar x \in \overline{\mathcal{R}(A)}$ is valid. Since $\,\xdag - \bar x \perp \mathcal{N}(A)\,$ implies that
$\,\xdag - \bar x \in \overline{\mathcal{R}(A^*)}\,$ (see, e.g.,~\cite[Fact~2.18 (iii)]{Bauschke11}), this condition however is a consequence of the identity $\overline{\mathcal{R}(A)}= \overline{\mathcal{R}(A^*)}$ for monotone operators $A \in \mathcal{L}(X)$ (see, e.g.,~\cite[Proposition~20.17]{Bauschke11}).
\end{proof}
We easily derive that for $\xdag-\bar{x}$ from Lemma~\ref{lem:lem2} and arbitrary $R>0$ there exist elements $w_R \in X$ with $\|w_R\| = R$ and $r_R \in X$ with $\|r_R\|=d(R)$ such that
an approximate source condition of the form
\begin{equation} \label{eq:apprsc}
\xdag-\bar{x}=Aw_R+r_R
\end{equation}
is valid. Then due to (\ref{eq:monprop}) we have
$$B^A_{\xdag}(\alpha)=\alpha\|(A+\alpha I)^{-1}(Aw_R+r_R)\|$$
$$ \le\alpha\|(A+\alpha I)^{-1}A\|\|w_R\|+\alpha\|(A+\alpha I)^{-1}\|\|r_R\|\le R \alpha +d(R),$$
and equilibrating the last two terms by means of the strictly decreasing auxiliary function
\begin{equation} \label{eq:Phi}
\Phi(R):=\frac{d(R)}{R}, \qquad 0<R<\infty, \qquad \lim \limits_{R \to \infty} d(R)=0,
\end{equation}
as $R:=\Phi^{-1}(\alpha)$ we have the assertion of the following proposition.
\begin{proposition} \label{pro:pro3}
Under the assumptions of Lemma~\ref{lem:lem2} we have the bias estimate
\begin{equation} \label{eq:lowlinbias}
B^A_{\xdag}(\alpha) \le 2 d(\Phi^{-1}(\alpha)), \qquad \alpha>0,
\end{equation}
for the Lavrentiev regularization in the linear case, where  $d(\Phi^{-1}(\alpha))$ is an index function with $\lim \limits_{\alpha \to 0}\frac{\alpha}{d(\Phi^{-1}(\alpha))}=0$. Moreover, for the a priori parameter choice $\alpha(\delta):=\Psi^{-1}(\delta)$ with $\Psi(\alpha):=\alpha \,d(\Phi^{-1}(\alpha))$ we have from (\ref{eq:triangle}) and (\ref{eq:noiseprop}) the estimate
$$ \|x_{\alpha(\delta)}^\delta-\xdag\| \le 3 d(\Phi^{-1}(\Psi^{-1}(\delta))),\qquad \delta>0,$$
for the total regularization error
and hence the convergence rate
$$\|x_{\alpha(\delta)}^\delta-\xdag\| =O(d(\Phi^{-1}(\Psi^{-1}(\delta)))) \qquad \mbox{as}\quad \delta \to 0.$$
\end{proposition}

We note that the assertion of Proposition~\ref{pro:pro3} remains valid if $d(R)$ beginning with (\ref{eq:Phi}) is replaced with a concave majorant of the distance function.

\begin{specialcase} \label{ex:ex1}
{\rm (Distance functions with power-type decay rate) As a consequence of the range identity
$$\mathcal{R}(A^p)= \mathcal{R}((AA^*)^{p/2}),$$
proven for all $0<p \le 1$ in \cite[Lemma~1]{PlMaHo16}, and on the basis of the assertion in \cite[Theorem~3.2]{DHY07} we have that for all $0<p<1$ the fractional power source conditions (\ref{eq:directp})
lead to distance functions (cf.~(\ref{eq:distfunct})) with a power-type decay as
\begin{equation} \label{eq:powerdist}
d(R) \le C\,R^{\frac{p}{p-1}}, \qquad C>0.
\end{equation}
The smaller $p>0$ the slower is the decay rate of $d(R) \to 0$ as $R \to \infty$ and the higher is for $\xdag-\bar{x}$ the degree of violation with respect to the benchmark source condition (\ref{eq:direct1}).
Applying Proposition~\ref{pro:pro3} this yields with $\Phi(R) \sim R^{\frac{1}{p-1}}$, $\Psi(\alpha) \sim \alpha^{p+1}$ and $\Psi^{-1}(\delta) \sim \delta^{\frac{1}{p+1}}$
the Lavrentiev regularization error convergence rates
$$B^A_{\xdag}(\alpha)=O(\alpha^p)  \qquad \mbox{and} \qquad \|x_{\alpha(\delta)}^\delta-\xdag\|=O(\delta^{\frac{p}{p+1}}) \quad \mbox{if}\quad \alpha(\delta) \sim \delta^\frac{1}{p+1}. $$
}\end{specialcase}

\begin{specialcase} \label{ex:ex2}
{\rm (Distance functions with logarithmic decay rate) If for $\xdag-\bar{x}$ the degree of violation with respect to the benchmark source condition (\ref{eq:direct1}) is so extreme that the power type decay (\ref{eq:powerdist}) cannot hold for arbitrarily small $p>0$,
then a very slow logarithmic decay rate
\begin{equation} \label{eq:logdist}
d(R) \le \frac{K}{(\log R)^q}, \qquad K,q >0,
\end{equation}
is still possible for sufficiently large $R \ge \underline R>0$.
Then the derived formula
\begin{equation} \label{eq:Rderfor}
B^A_{\xdag}(\alpha) \le  R \alpha +d(R)
\end{equation}
applies by setting $R:=\alpha^{-\kappa},\;0<\kappa<1$, and provides us with the estimate
$$ B^A_{\xdag}(\alpha) \le \alpha^{1-\kappa}+ \frac{K}{\left[\log ( \alpha^{-\kappa})\right]^q}=\alpha^{1-\kappa}+ \frac{K}{\kappa^q\left[\log \left( \frac{1}{\alpha}\right)\right]^q} $$
and owing to $\alpha^{1-\kappa}=O\left(\left[\log \left( \frac{1}{\alpha}\right)\right]^q \right)$ as $\alpha \to 0$  we have
\begin{equation} \label{eq:biaslog}
 B^A_{\xdag}(\alpha)= O\left(\left[\log \left( \frac{1}{\alpha}\right)\right]^q \right) \quad \mbox{as}\quad \alpha \to 0.
\end{equation}
 This logarithmic convergence rate for the bias also leads to a logarithmic rate for the noisy data case of linear Lavrentiev regularization. Namely, we derive from  (\ref{eq:biaslog}) in combination with
 (\ref{eq:triangle}) and (\ref{eq:noiseprop}) the convergence rate
 \begin{equation} \label{eq:totalog}
 \|x^\delta_{\alpha(\delta)}-\xdag\|= O\left(\left[\log \left( \frac{1}{\delta}\right)\right]^q \right) \quad \mbox{as}\quad \delta \to 0
\end{equation}
when the regularization parameter is chosen a priori as $\alpha(\delta) \sim \delta^\zeta$ with exponent $0<\zeta<1$.
The very low logarithmic convergence rates of the form (\ref{eq:totalog}) are well-known in regularization theory (see, e.g, \cite{BoHo10,HohWei15,Kalten08,WeHo12}) and inevitable if the solution is not smooth enough with respect to the forward operator.
}\end{specialcase}

\begin{remark} \label{rem:Rformel}
{\rm Under the conditions $\xdag-\bar x \notin \mathcal{R}(A)$ and $\xdag-\bar x \perp \mathcal{N}(A)$  the properties of the distance function $d(R)$ from (\ref{eq:distfunct}) are such that the inequality (\ref{eq:Rderfor}) is valid for sufficiently large $R \ge \underline R>0$. Then the function  $\Phi(R)$ defined by formula (\ref{eq:Phi}) is strictly decreasing for those $R$ and tends to zero as $R \to \infty$. Moreover, the mapping
$R \mapsto \alpha:=\Phi(R)$ is injective (strictly decreasing) and well-defined for $\underline R \le R < \infty$ with $\alpha \to 0$ as $R \to \infty$. Then  it is clear that the function $d(\Phi^{-1}(\alpha))$ is
well-defined and strictly increasing for sufficiently small $\alpha>0$ with the limit condition $\lim \limits_{\alpha \to 0} d(\Phi^{-1}(\alpha))=0$. This, however yields the bias convergence
$$\lim \limits_{\alpha \to 0} B_{\xdag}^A(\alpha)=0$$
under the above stated conditions on $\xdag-\bar x$.
}\end{remark}

\begin{example}\label{ex:M}
{\rm (Multiplication operator) We start this list of examples with the simple multiplication operator $A:=M$ in the real Hilbert space $X:=L^2(0,1)$ generated by a real multiplier function $m \in L^\infty(0,1)$, which is defined as
 \begin{equation} \label{eq:M}
[Mx](t):= m(t)\,x(t), \qquad 0 \le t \le 1,
\end{equation}
where we restrict to continuous and strictly increasing functions $m$ with $\lim \limits _{t \to 0} m(t)=0$ and $\lim \limits _{t \to 1} m(t)=1$. Then the linear operator $M$ is bounded, injective, non-compact
with continuous spectrum $\sigma(M)=[0,1]$, selfadjoint and monotone, i.e.~the associated linear operator equation (\ref{eq:linopeq}) is ill-posed. The smoothness of the solution $\xdag$ with respect to $M$ and $\bar x$ and the decay rate of
the distance function $d(R)$ if the benchmark source condition is violated, i.e.~if $(\xdag-\bar x)/m \notin L^2(0,1)$, will be determined by the decay rate of $m(t) \to 0$ as $t \to 0$. For example, in the situation $\xdag-\bar x \equiv 1$ we have a logarithmic decay of the distance function $d(R) \le K/\log(R)$ for sufficiently large $R$ and hence logarithmic
convergence rates (\ref{eq:totalog}) with $q=1$ for $m(t)=\exp(1-1/\sqrt{t})$ (see for details \cite[Example~3]{FHM11}).
}\end{example}

\begin{example} \label{ex:V}
{\rm (Fractional integral operator) As second point in this series we present in the real Hilbert space $X:=L^2(0,1)$ the Riemann-Liouville fractional integral operator (cf.~\cite{GorLuYam15}) $A:=V$, also called Volterra operator (cf.~\cite{Haase06}), defined by
 \begin{equation} \label{eq:Volterra}
[Vx](s):=\int \limits_0^s x(t) dt, \qquad 0 \le s \le 1.
\end{equation}
This is an example which helps to distinguish different situations of solution smoothness
with respect to the forward operator. The linear operator $V$ is bounded, injective, compact, non-selfadjoint and monotone with $0 \in \sigma(V)$, i.e.~the associated linear operator equation (\ref{eq:linopeq}) is ill-posed. On the hand, we have situations, where $\xdag$ allows for H\"older source conditions (\ref{eq:directp}) implying H\"older rates (\ref{eq:ratep}) for the bias of the Lavrentiev
regularization. On the other hand, there exist elements $\xdag$ with even less smoothness such that (\ref{eq:directp}) fails for arbitrarily small exponents $p>0$. Then only lower rates, for example logarithmic rates (\ref{eq:biaslog}), remain for the bias taking into account (cf.~Remark~\ref{rem:rem1}) that we have for all $\xdag \in X$ with $\xdag - \bar x \perp \mathcal{N}(A)$ the limit condition  $B^A_{\xdag}(\alpha) \to 0$ as $\alpha \to 0$ and hence a (perhaps very low) well-defined bias rate for that specific $\xdag$.

From \cite[Theorem~2.1]{GorLuYam15} (see also \cite{GorYam99}) we take the explicit structure of the ranges of the fractional powers of the integral operator (\ref{eq:Volterra}) in terms of fractional order Sobolev-Hilbert spaces $H^p(0,1)$ as
\begin{equation} \label{eq:Vrange1}
\mathcal{R}(V^p)=\left\{\begin{array}{ccc} H^p[0,1] &\mbox{for}& 0<p<\frac{1}{2}\\ \{u \in H^{\frac{1}{2}}[0,1]:\int\limits_0^1 \frac{|u(t)|^2}{t}dt<\infty \}&\mbox{for}& p=\frac{1}{2}
\\ \{u \in H^p[0,1]:\, u(0)=0\}   & \mbox{for}& \frac{1}{2}<p \le 1 \end{array} \right.
\end{equation}
Moreover, due to the injectivity of the Volterra operator $V$ in $L^2(0,1)$, the orthogonality condition $\xdag - \bar x \perp \mathcal{N}(A)$ and hence the bias limit $B^V_{\xdag}(\alpha) \to 0$ as $\alpha \to 0$ are trivially satisfied.
Now, for all $0<p < 1/2$, H\"older source conditions (\ref{eq:directp}) hold true for $A:=V$ if and only if $\xdag-\bar{x} \in H^p(0,1)$. For $1/2 \le p \le 1$ additional conditions on $\xdag-\bar{x}$ have to be imposed.
Then, with that non-selfadjoint monotone forward operator $V$, a necessary condition for the situation that a logarithmic bias rate (\ref{eq:biaslog}) is valid, but (\ref{eq:ratep}) fails for arbitrarily small $p>0$, is
$$\xdag,\bar{x} \in L^2(0,1), \quad \mbox{but} \quad \xdag-\bar{x} \not \in H^p(0,1) \quad \mbox{for arbitrarily small} \quad p>0.$$
}\end{example}

\begin{example}\label{ex:C}
{\rm (Ces\`aro operator) As third example we consider in $X:=L^2(0,1)$ the continuous version of the Ces\`aro operator $A:=C$ (cf.~\cite[p.~133]{BroHalShi65}) defined as
 \begin{equation} \label{eq:C}
[Cx](s):=\frac{1}{s}\int \limits_0^s x(t) dt, \qquad 0 \le s \le 1,
\end{equation}
which is, in contrast to (\ref{eq:Volterra}), an injective, monotone, but non-compact linear operator. Again we have $0 \in \sigma(C)$ and hence ill-posedness, but the conditions $\xdag-\bar x \in \mathcal{R}(C^p)$ for
exponents $0<p \le 1$  are not so easy to check like in Example~\ref{ex:V} in order to derive H\"older convergence rates (\ref{eq:linHoelder}). This is a good example for proving the capability of distance functions for
that purpose. Evidently, $\xdag-\bar x \equiv 1$ satisfies the benchmark source condition (\ref{eq:direct1}) with source element $w \equiv 1$, but the Heaviside-type function
\begin{equation} \label{eq:exampleC}
\xdag(t)-\bar x(t)= \left\{\begin{array}{ccc} 0 &\mbox{if}& 0 \le t < 1/2\\ 1 & \mbox{if}& 1/2 \le t \le 1 \end{array} \right.
\end{equation}
fails to satisfy (\ref{eq:direct1}) for arbitrary $w \in L^2(0,1)$. However, as Proposition~\ref{pro:C} will show, we have $d(R) \le \frac{K}{R}$ for some constant $K>0$ and sufficiently large $R>0$, which yields the inequality (\ref{eq:powerdist})
with $p=1/2$ and hence the H\"older rate $\mathcal{O}(\delta^{1/3})$ for the Lavrentiev regularization.
}\end{example}
\begin{proposition} \label{pro:C}
For $\xdag-\bar x$ from (\ref{eq:exampleC}) we have for some constant $K>0$ and sufficiently large $R>0$ the estimate
$$d(R)\,=\,\min\{\|\xdag-\bar{x}-Cw\| : \,\|w\|\leq R\} \,\le\, \frac{K}{R}\,.$$
\end{proposition}
\begin{proof}
For the function  $\;w_R(t)= \left\{\begin{array}{ccc} 0 &\mbox{if}& 0 \le t < 1/2\\R^2 & \mbox{if}& 1/2 \le t \le 1/2+1/(2R^2) \\ 1 & \mbox{if}& 1/2+1/(2R^2) < t \le 1 \end{array} \right.\;$ we have
$$\|w_R\| \le R \quad \mbox{and} \quad d^2(R) \le \int \limits_0^1 \left( \frac{1}{s}\int \limits_0^s w_R(t) dt - (\xdag(s)-\bar x(s)) \right)^2 ds\,. $$
Furthermore, because $w_R$ and $\xdag-\bar x$ are identically zero on $[0,1/2)$ we can estimate as
\begin{align*}
d^2(R) & \le \int \limits_0^1 \frac{1}{s ^2}\left(\int \limits_0^s w_R(t) dt - s(\xdag(s)-\bar x(s)) \right)^2 ds \\
& \le 4 \int \limits_{1/2}^1 \left( \int \limits_{1/2}^s w_R(t) dt - s(\xdag(s)-\bar x(s))\right)^2 ds\\
& = 4\left[\int \limits_0^{1/(2R^2)}\left(s(R^2-1)-\frac{1}{2}\right)^2ds + \frac{1}{4R^4} \left(\frac{1}{2} - \frac{1}{2R^2} \right)  \right]\\
& = 4\left[\frac{(R^2-1)^2}{24R^6}-\frac{R^2-1}{8R^4}+\frac{1}{8R^2}+\frac{1}{8R^4} - \frac{1}{8R^6}\right] \le K^2\,\frac{1}{R^2}
\end{align*}
 for sufficiently large $R>0$. This completes the proof.
\end{proof}

\section{Extensions to the nonlinear case under specific conditions on the nonlinearity structure}\label{s5}
\setcounter{equation}{0}
\setcounter{theorem}{0}

Now we return to the Lavrentiev regularization for nonlinear operator equations (\ref{eq:nonlinopeq}) with regularized solutions $\xad$ satisfying for noisy data $\yd$ the singularly perturbed equation (\ref{eq:lav})
and with regularized solutions $\xa$ satisfying (\ref{eq:Lavnonlin0}) in the noise-free case. We are going to handle the corresponding nonlinear bias $B^F_{\xdag}(\alpha):=\|\xa-\xdag\|$ in the noise-free case, where we try to incorporate experiences from the studies of the linear case in Section~\ref{s4}. By Lemma~\ref{lem:lem1} the properties of the bias $B^F_{\xdag}(\alpha)$ allow us immediately to derive the essential behaviour of the
total error $\|\xad-\xdag\|$ for the Lavrentiev regularization also in the nonlinear case.

Throughout this section let us suppose that the following assumption holds in addition to Assumption~\ref{ass:ass0}.
\begin{assumption} \label{ass:ass2}
\begin{enumerate}\item[]
\item[(i)] Let there exist a ball $\mathcal{B}_r(\xdag) \subset X$ around a solution $\xdag$ to equation (\ref{eq:nonlinopeq}) with sufficiently large radius
\begin{equation} \label{eq:r}
r > \|\xdag-\bar x\|
\end{equation}
such that $F$ is Fr\'echet differentiable in the ball with Fr\'echet derivatives $F^\prime(x) \in \mathcal{L}(X)$ and the mapping $x \mapsto F^\prime(x)$ is continuous at every $x \in \mathcal{B}_r(\xdag)$.
\item[(ii)] Let there exist a constant $k_0>0$ and a function $g$ such that, for every $\tilde x,x \in \mathcal{B}_r(\xdag)$ and $v \in X$, there is $g(\tilde x,x,v) \in X$ satisfying the nonlinearity condition
\begin{equation} \label{eq:specnlcond}
(F^\prime(\tilde x)-F^\prime(x))\,v=F^\prime(x)\,g(\tilde x,x,v), \qquad \|g(\tilde x,x,v)\|\le k_0\,\|\tilde x-x\|\,\|v\|.
\end{equation}
\end{enumerate}
\end{assumption}

Item (ii) of Assumption~\ref{ass:ass2} occurs in numerous papers on regularization theory in a more or less modified form, but we follow the precise ideas in \cite{MaNa13} which have filled gaps of the previous literature
(see for details \cite[p.~195]{MaNa13})). Furthermore, we mention at this point that example classes of nonlinear forward operators $F$ satisfying the specific nonlinearity condition (\ref{eq:specnlcond}) were presented, for example, in \cite{SEK93} and  \cite{Argyros13,BurKal06}.

Now we are ready to formulate the proposition of this section, which assert that under the assumed nonlinearity conditions and for a sufficiently good reference element $\bar x$ the bias $B^F_{\xdag}(\alpha)$ in the nonlinear case
is proportional to and hence fully determined by the bias function $B^A_{\xdag}(\alpha)$, where $A$ is the monotone operator $F^\prime(\xdag) \in \mathcal{L}(X)$. Thus, convergence as well as H\"older or logarithmic rates for the noise-free bias  $B^A_{\xdag}(\alpha)$ as $\alpha \to 0$ carry over to the same rates for the bias $B^F_{\xdag}(\alpha)$ with corresponding consequences for the convergence and for the rates of the total regularization errors $\|\xad-\xdag\|$ as $\delta \to 0$.

\begin{proposition} \label{thm:main}
Under Assumptions~\ref{ass:ass0} and \ref{ass:ass2} with the additional condition \linebreak $k_0\,\|\xdag-\bar x\|<2$  let the operator $A:=F^\prime(\xdag)$ denote the Fr\'echet derivative of the monotone nonlinear operator $F$ at the solution $\xdag$ to equation (\ref{eq:nonlinopeq}). Then we have, for all $\alpha>0$ and associated elements $\xa$ solving (\ref{eq:Lavnonlin0}), the inequality
\begin{equation} \label{eq:upbound}
B^F_{\xdag}(\alpha) \le \|\xdag-\bar x\|
\end{equation}
and thus from (\ref{eq:r}) the condition $\xa \in \mathcal{B}_r(\xdag)$.
Furthermore, we have the estimate
\begin{equation} \label{eq:estthm0}
B^F_{\xdag}(\alpha) \le C\,B^A_{\xdag}(\alpha), \qquad \mbox{with} \qquad C=\frac{2+2k_0\,\|\xdag-\bar{x}\|}{2-k_0\,\|\xdag-\bar x\|}\,.
\end{equation}
Hence, we have for all $\alpha>0$ and $\delta \ge 0$ the estimate
\begin{equation} \label{eq:estthm}
\|\xad-\xdag\| \le  C\,B^A_{\xdag}(\alpha)+\frac{\delta}{\alpha}
\end{equation}
for the total regularization error.
\end{proposition}
\begin{proof}
First we quote as result from \cite[Proposition~3.1]{Taut02} that the inequality (\ref{eq:upbound}) and by (\ref{eq:r}) also $\xa \in \mathcal{B}_r(\xdag)$ are valid for all $\alpha>0$.
It is well-known that the Fr\'echet derivatives $A:=F^\prime(\xdag)$ and $A_\alpha:=F^\prime(\xa)$ are monotone bounded linear operators mapping in $X$ if the nonlinear operator $F$ is monotone. Now, amending some ideas
along the lines of the proofs of \cite[Proposition~2.2]{MaNa13} and \cite[Proposition~3.3]{Taut02} we derive from (\ref{eq:Lavnonlin0})  that
$$\xa-\xdag=\bar x-\xdag+(A_\alpha+\alpha I)^{-1}[y-F(\xa)+A_\alpha(\xa-\bar x)]=u_\alpha+v_\alpha+w_\alpha, $$
with the three terms
$$u_\alpha:=\alpha(A+\alpha I)^{-1}(\bar x-\xdag), \qquad \mbox{where} \qquad \|u_\alpha\|=B_{\xdag}^A(\alpha),$$
$$v_\alpha:= (A_\alpha+\alpha I)^{-1}(A-A_\alpha)\,u_\alpha\,,$$
and
$$w_\alpha:=(A_\alpha+\alpha I)^{-1}[F(\xdag)-F(\xa)+A_\alpha(\xa-\xdag)]  \,.$$
The handling of the term  $v_\alpha$ is based on the nonlinearity condition (\ref{eq:specnlcond}), where we have with $(A-A_\alpha)\,u_\alpha=A_\alpha\,g(\xdag,\xa,u_\alpha)$ the estimate
$$\|v_\alpha\|=\|(A_\alpha+\alpha I)^{-1}A_\alpha\,g(\xdag,\xa,u_\alpha\| \le k_0\,\|\xa-\xdag\|\|u_\alpha\|\,. $$
Owing to the continuous Fr\'echet differentiability of $F$ in the ball $\mathcal{B}_r(\xdag)$ the fundamental theorem of calculus (mean value theorem in integral form) applies to estimate from above the norm  $\|w_\alpha\|$ of the third term. Precisely, we have
$$F(\xdag)-F(\xa)-A_\alpha(\xdag-\xa)=\int \limits_0^1 [F^\prime(\xa+t(\xdag-\xa))-A_\alpha](\xdag-\xa)\,dt$$
$$=A_\alpha\,\int \limits_0^1 g(\xa+t(\xdag-\xa),\xa,\xdag-\xa)\,dt$$
and
$$ \| g(\xa+t(\xdag-\xa),\xa,\xdag-\xa)\| \le k_0\,\|\xa-\xdag\|^2\,t\,.$$
Then we have due to the second inequality in (\ref{eq:monprop})
$$\|w_\alpha\|=\|(A_\alpha+\alpha I)^{-1}A_\alpha\int \limits_0^1 g(\xa+t(\xdag-\xa),\xa,\xdag-\xa)\,dt\| \le \frac{k_0}{2}\,\|\xa-\xdag\|^2\,.$$
Summarizing the results for $u_\alpha$, $v_\alpha$ and $w_\alpha$ we obtain
$$\|\xa-\xdag\|\le \left(1+k_0\,\|\xa-\xdag\|\right)\|u_\alpha\|+\frac{k_0}{2}\,\|\xa-\xdag\|^2$$
and taking into account (\ref{eq:upbound})
$$\|\xa-\xdag\|\le  \left(1+k_0\,\|\xdag-\bar x\|\right)\|u_\alpha\|+\frac{k_0}{2}\,\|\xdag-\bar x\|\,\|\xa-\xdag\|\,. $$
This yields (\ref{eq:estthm0}) as
$$B^F_{\xdag}(\alpha)=\|\xa-\xdag\| \le\left(\frac{2+2k_0\,\|\xdag-\bar{x}\|}{2-k_0\,\|\xdag-\bar x\|}\right)\,\|u_\alpha\|=\left(\frac{2+2k_0\,\|\xdag-\bar{x}\|}{2-k_0\,\|\xdag-\bar x\|}\right)\,B^A_{\xdag}(\alpha). $$
The estimate (\ref{eq:estthm}) is a consequence of Lemma~\ref{lem:lem1} (cf.~formula (\ref{eq:noiseprop})) in combination with the triangle inequality (\ref{eq:triangle}). This completes the proof.
\end{proof}

The following corollary from Proposition~\ref{thm:main} is a counterpart to Proposition~\ref{pro:pro10} concerning the convergence of the bias as the regularization parameter $\alpha$ tends to zero.

\begin{corollary} \label{cor:cor1}
Under the assumptions of Proposition~\ref{thm:main} we have
$$\lim \limits _{\alpha \to 0} B_{\xdag}^F(\alpha) =0$$
if  $$\xdag-\bar x \perp \mathcal{N}(F^\prime(\xdag)).$$
\end{corollary}

\begin{remark} \label{rem:rem2}
{\rm The very specific nonlinearity condition (\ref{eq:specnlcond}) can be avoided if the benchmark source condition in the nonlinear case
$$ \xdag-\bar x= F^\prime(\xdag)\,w, \quad w \in X,$$
applies together with the simpler Lipschitz continuity
$$ \|F^\prime(x)-F^\prime(\xdag)\| \le L\,\|x-\xdag\| \qquad \mbox{for all} \qquad x \in \mathcal{B}_r(\xdag)$$
as nonlinearity condition for some $L>0$. Then we have from \cite[Theorem 3.2]{Taut02} the nonlinear bias estimate
$$ B^F_{\xdag}(\alpha) \le \left(\|w\|+\frac{L}{2}\|w\|^2 \right)\,\alpha,$$
which yields together with (\ref{eq:triangle}) and (\ref{eq:noiseprop}) the convergence rate
$$\|x_{\alpha(\delta)}^\delta-\xdag\| =\mathcal{O}(\sqrt{\delta}) \qquad \mbox{as} \qquad \delta \to 0$$
for the total regularization error if the regularization parameter is chosen as $\alpha(\delta) \sim\sqrt{\delta}$.

A convergence rate result similar to that of Proposition~\ref{thm:main} for the Lavrentiev regularization of nonlinear operator equations was also presented as Theorem~8 in \cite{HKR16}.
In contrast to item (ii) of our Assumption~\ref{ass:ass2} a range invariance occurs there as structural condition of nonlinearity for the forward operator $F$, which provides the opportunity to use a monotone operator $A \in \mathcal{L}(X)$
different from the Fr\'echet derivative $F^\prime(\xdag)$. However, the cross connections between the linear bias $B^A_{\xdag}(\alpha)$ and its nonlinear counterpart $B^F_{\xdag}(\alpha)$ are not so clear in \cite{HKR16}
as they are in Proposition~\ref{thm:main}.
}\end{remark}

\section*{Acknowledgments}
The authors are grateful to Robert Plato (Siegen) for valuable hints and private communications.
Research of BH was partially supported by the German Research Foundation (DFG) under grant HO~1454/8-2.


\begin{thebibliography}{10}

\bibitem{AlbRya06}
Y.~Alber and I.~Ryazantseva:
{\em Nonlinear Ill-posed Problems of Monotone Type}, Springer, Dordrecht, 2006.

\bibitem{AndreevEtal15}
R.~Andreev, P.~Elbau, M.~V.~de Hoop, L.~Qiu, and O.~Scherzer:
Generalized convergence rates results for linear inverse
              problems in {H}ilbert spaces. {\em Numer.~Funct.~Anal.~Optim.}
{\bf 36}(5), 549--566 (2015).


\bibitem{Argyros13}
I.~K.~Argyros, Y.~J.~Cho and S.~George:
Expanding the applicability of Lavrentiev regularization
methods for ill-posed problems.
{\em Bound.~Value Probl.} {\bf 2013}, 114 (15pp) (2013).

\bibitem{Argyros14}
I.~K.~Argyros and S.~George: Expanding the applicability of Lavrentiev regularization methods for ill-posed equations
under general source condition.
{\em Nonlinear Funct.~Anal.~Appl.} {\bf 19}(2), 177--192 (2014).


\bibitem{Bauschke11}
H.~H.~Bauschke and P.~L.~Combettes:
{\em Convex Analysis and Monotone Operator Theory in Hilbert Spaces},
Springer, New York, 2011.


\bibitem{Bes10}
A.~Besenyei: On uniformly monotone operators arising in nonlinear elliptic
and parabolic problems. {\em Ann.~Univ.~Sci.~Budapest. E\"otv\"os Sect. Math.}
{\bf 53}, 33--43 (2010).


\bibitem{BoHo10}
R.~I.~Bo\c{t} and B.~Hofmann: An extension of the variational inequality approach for obtaining
convergence rates in regularization of nonlinear ill-posed problems.
{\em Journal of Integral Equations and Applications} {\bf 22}(3), 369--392 (2010).


\bibitem{BroHalShi65}
A.~Brown, A. P.~R.~Halmos,  and A.~L.~Shields:
Ces{\`a}ro operators. {\em Acta Sci.~Math.~(Szeged)} {\bf 26}, 125--137 (1965).

\bibitem{BurKal06}
M.~Burger and B.~Kaltenbacher:
Regularizing Newton-Kaczmarz methods for nonlinear ill-posed problems.
{\em SIAM J.~Numer.~Anal.} {\bf 44}(1), 153--182 (2006).


\bibitem{ChengHofLu14}
J.~Cheng, B.~Hofmann, and S.~Lu:
The index function and {T}ikhonov regularization for ill-posed
problems. {\em J.~Comput.~Appl.~Math.} {\bf 265}, 110--119 (2014).

\bibitem{ChengYam00}
J.~Cheng and M.~Yamamoto: One new strategy for a priori choice of regularizing
              parameters in {T}ikhonov's regularization. {\em Inverse Problems}
              {\bf 16}(4), L31--L38 (2000)

\bibitem{DHY07}
D.~D{\"u}velmeyer, B.~Hofmann, and M.~Yamamoto: Range inclusions and approximate source conditions with general
  benchmark functions. {\em Numer. Funct. Anal. Optim.} {\bf 28}(11-12), 1245--1261 (2007).


\bibitem{EHN96}
H.~W. Engl, M.~Hanke, and A.~Neubauer: {\em Regularization of Inverse Problems}.
Kluwer Academic Publishers,  Dordrecht, 1996, 2nd Edition 2000.


\bibitem{FHM11}
J.~Flemming, B.~Hofmann, and P.~Math{\'e}: Sharp converse results for the regularization error using distance
  functions. {\em Inverse Problems}, {\bf 27}(2), 025006 (18pp) (2011).

\bibitem{GorLuYam15}
R.~Gorenflo and Yu.~Luchko and M.~Yamamoto: Time-fractional diffusion equation in the fractional Sobolev spaces.
{\em Fract.~Calc.~Appl.~Anal.} {\bf 18}(3), 799--820 (2015).


\bibitem{GorYam99}
R.~Gorenflo and M.~Yamamoto:
Operator-theoretic treatment of linear Abel integral equations of first kind.
{\em Japan J.~Indust.~Appl.~Math.} {\bf 16}(1), 137--161 (1999).


\bibitem{Groe84}
C.~W.~Groetsch:
{\em The Theory of Tikhonov Regularization for Fredholm Equations of the First Kind},
Pitman, Boston, MA, 1984.

\bibitem{Haase06}
M.~Haase: {\em The Functional Calculus for Sectorial Operators}, Operator Theory: Advances and Applications,
Vol.~169, Birkh\"auser Verlag, Basel, 2006.

\bibitem{HeinHof09}
T.~Hein and B.~Hofmann:
Approximate source conditions for nonlinear ill-posed
problems---chances and limitations. {\em Inverse Problems} {\bf 25}(3), 035003 (16pp) (2009).

\bibitem{HohWei15}
T.~Hohage and F.~Weidling: Verification of a variational source condition for acoustic inverse medium scattering problems.
{\em Inverse Problems} {\bf 31}(7), 075006 (14pp) (2015).

\bibitem{Hof06}
B.~Hofmann: Approximate source conditions in Tikhonov-Phillips regularization
  and consequences for inverse problems with multiplication operators.
{\em Math. Methods Appl. Sci.} {\bf 29}(3), 351--371 (2006).

\bibitem{Hof15}
B.~Hofmann.
On smoothness concepts in regularization for nonlinear inverse problems
in Banach spaces. Chapter~8 in {\em Mathematical and Computational Modeling: With Applications in Natural and Social Sciences, Engineering, and the Arts} (Ed.: R.~Melnik).
\newblock John Wiley, New Jersey 2015,  pp.~192--221.

\bibitem{HKR16}
B.~ Hofmann, B.~Kaltenbacher, and E.~Resmerita:
Lavrentiev's regularization method in {H}ilbert spaces revisited. \url{https://arxiv.org/abs/1506.01803v2}.
To appear in {\em Inverse Probl.~Imaging} {\bf 10}(3) (2016).

\bibitem{HofMat07}
B.~Hofmann and P.~Math{\'e}: Analysis of profile functions for general linear regularization
  methods. {\em SIAM J.~Numer.~Anal.} {\bf 45}(3), 1122--1141 (2007).

\bibitem{HofSch94}
B.~Hofmann and O.~Scherzer: Factors influencing the ill-posedness of nonlinear problems.
{\em Inverse Problems} {\bf 10}(6), 1277--1297 (1994).

\bibitem{HofYam10}
B.~Hofmann and M.~Yamamoto: On the interplay of source conditions and variational
              inequalities for nonlinear ill-posed problems.
{\em Appl.~Anal.} {\bf 89}(11), 1705--1727 (2010).


\bibitem{Kalten08}
B.~Kaltenbacher: A note on logarithmic convergence rates for nonlinear Tikhonov regularization.
{\em J.~Inv.~Ill-Posed Prob.} {\bf 16}(1), 79--88 (2008).

\bibitem{Lavrentiev67}
M.~M.~Lavrentiev:
{\em Some Improperly Posed Problems of Mathematical Physics}.
Springer, New York, 1967.


\bibitem{LiuNash96}
F.~Liu and M.~Z.~Nashed: Convergence of regularized solutions of nonlinear ill-posed
              problems with monotone operators. In:
{\em Partial Differential Equations and Applications}, Lecture Notes in Pure and Appl.~Math.,
Vol.~177, pp.~353--361. Dekker, New York, 1996.

\bibitem{MaNa13}
P.~Mahale and M.~T.~Nair: Lavrentiev regularization of nonlinear ill-posed equations
              under general source condition. {\em J.~Nonlinear Anal.~Optim.} {\bf 4}(2),
              193--204 (2013).

\bibitem{MatHof08}
P.~Math{\'e} and B.~Hofmann: How general are general source conditions?
{\em Inverse Problems} {\bf 24}(1), 015009 (5pp) (2008).

\bibitem{NairTau04}
M.~T.~Nair and U.~Tautenhahn:
Lavrentiev regularization for linear ill-posed problems under
              general source conditions.
{\em Z.~Anal.~Anwendungen} {\bf 23}(1), 176--185 (2004).

\bibitem{Nashed87}
M.~Z.~Nashed: A new approach to classification and regularization of ill-posed
  operator equations. In: {\sl Inverse and Ill-posed Problems} (Sankt Wolfgang, 1986),
  volume~4 of Notes Rep.~Math.~Sci.~Engrg. Academic Press, Boston, MA, 1987, pp.~53--75.


\bibitem{Plato95}
R.~Plato: {\em Iterative and Other Methods for Linear Ill-Posed Equations}.
Habilitation Thesis, Techn.~Univ. Berlin, 1995.

\bibitem{Plato16}
R.~Plato:  Converse results, saturation and quasi-optimality for
      Lavrentiev regularization of accretive problems.
\url{https://arxiv.org/abs/1607.04879v1}, 2016.



\bibitem{PlMaHo16}
R.~Plato, P.~Math\'e, and B. Hofmann: Optimal rates for Lavrentiev regularization with
adjoint source conditions. Preprint 2016-03, Preprintreihe der Fakult\"at f\"ur Mathematik, TU Chemnitz, Germany, 2016.
\url{http://nbn-resolving.de/urn:nbn:de:bsz:ch1-qucosa-199010}.


\bibitem{SEK93}
O.~Scherzer,  H.~W.~Engl, and K.~Kunisch:
Optimal a posteriori parameter choice for Tikhonov
regularization for solving nonlinear ill-posed problems.
{\em SIAM J.~Numer.~Anal.} {\bf 30}(6), 1796--1838 (1993).


\bibitem{Schuster12}
T.~Schuster, B.~Kaltenbacher, B.~Hofmann, and K.~S.~Kazimierski:
{\em Regularization Methods in Banach Spaces}. Walter de Gruyter, Berlin/Boston, 2012.



\bibitem{Seme10}
E.~V.~Semenova: Lavrentiev regularization and balancing principle for solving
              ill-posed problems with monotone operators. {\em Comput.~Methods Appl.~Math.} {\bf 10}(4), 444--454 (2010).



\bibitem{Taut02}
U.~Tautenhahn: On the method of Lavrentiev regularization for nonlinear
ill-posed problems. {\em Inverse Problems} {\bf 18}(1), 191--207 (2002).


\bibitem{WeHo12}
F.~Werner and T.~Hohage:
Convergence rates in expectation for Tikhonov-type regularization of inverse problems with Poisson data.
 {\em Inverse Problems} {\bf 28}(10), 104004 (15pp.) (2012).

\bibitem{Zeidler90}
E.~Zeidler: {\em Nonlinear Functional Analysis and its Applications - II/B: Nonlinear Monotone Operators}. Springer, 1990.

\end{thebibliography}
\end{document}